\documentclass[11pt]{article}
\usepackage{enumerate}
\usepackage{amssymb,a4wide,latexsym,makeidx,epsfig,fleqn}
\usepackage{amsthm}
\usepackage{amsmath}
\usepackage{enumerate}
\usepackage{extarrows}
\usepackage{graphicx}
\usepackage{subfigure}
\usepackage{float}
\usepackage{cleveref}
\usepackage[justification=centering]{caption}
\newtheorem{theorem}{Theorem}[section]

\newtheorem{lemma}[theorem]{Lemma}

\begin{document}
\textwidth 150mm \textheight 225mm
\title{On the $D_\alpha$ spectral radius of non-transmission regular graphs
\thanks{Supported by the National Natural Science Foundation of China (Nos. 12001434 and 12271439).}}
\author{Zengzhao Xu$^{a}$, Weige Xi$^{a}$\footnote{Corresponding author.}, Ligong Wang$^{b}$\\
{\small $^{a}$ College of Science, Northwest A\&F University, Yangling, Shaanxi 712100, P.R. China}\\
{\small $^{b}$ School of Mathematics and Statistics, Northwestern Polytechnical University,}\\
{\small  Xi'an, Shaanxi 710129, P.R. China}\\
{\small E-mail: xuzz0130@163.com; xiyanxwg@163.com; lgwangmath@163.com}\\}

\date{}
\maketitle
\begin{center}
\begin{minipage}{120mm}
\vskip 0.3cm
\begin{center}
{\small {\bf Abstract}}
\end{center}
{\small  Let $G$ be a connected graph with order $n$ and size $m$. Let $D(G)$ and $Tr(G)$ be
the distance matrix and diagonal matrix with vertex transmissions of $G$, respectively.
For any real $\alpha\in[0,1]$, the generalized distance matrix $D_\alpha(G)$ of $G$ is defined as
$$D_\alpha(G)=\alpha Tr(G)+(1-\alpha)D(G).$$
The largest eigenvalue of $D_{\alpha}(G)$ is called the $D_{\alpha}$ spectral radius or generalized distance spectral radius of $G$, denoted by $\mu_{\alpha}(G)$. In this paper, we establish a lower bound on the difference between the maximum vertex transmission and the $D_\alpha$ spectral radius of non-transmission regular graphs, and we also characterize the extremal graphs attaining the bound.

}
\end{minipage}
\end{center}

\section{Introduction }
\label{sec:ch6-introduction}

Let $G=(V(G),E(G))$ be a graph with vertex set $V(G)=\{v_1,v_2,$ $\ldots, v_n\}$ and edge set $E(G)$. The number of vertices $|V(G)|=n$ is the order of $G$ and the number of edges $|E(G)|=m$ is the size of $G$. For any vertex $v_i\in V(G)$, $N(v_i)=\{v_j : (v_i,v_j)\in E(G) \}$ is called the set of neighbors of $v_i$. The degree $d(v_i)$ of a vertex $v_i$ in $G$ is the number of edges of $G$ incident with $v_i$, i.e. , $d(v_i)=|N(v_i)|$. Let $K_n$ denote the complete graph of order $n$. Throughout this paper all graphs considered are simple and connected.

The distance between two vertices $v_i$ and $v_j$, denoted by $d(v_i,v_j)$ or simply by $d_{ij}$, is the length of a shortest path from $v_i$ to $v_j$ in $G$. The diameter of $G$, denoted by $diam(G)$, is the maximum distance between any two vertices of $G$. For any vertex $v_i\in V(G)$, the transmission of $v_i$ in $G$, denoted by $Tr_i$, is the sum of distances from $v_i$ to all other vertices of $G$, that is, $Tr_i=\sum\limits_{k=1}^n d_{ik}$. We use $Tr_{max}(G)$ to denote the maximum vertex transmission of $G$  (or simply $Tr_{max}$ if it is clear from the context). For a connected graph $G$, if $Tr_1=Tr_2=\cdots=Tr_n=r$, we say that $G$ is $r$-transmission regular; otherwise, $G$ is non-transmission regular. The sum of all distances between all unordered pairs of vertices in $G$ is called the Wiener index of graph $G$, denoted by $W(G)$, i.e., $W(G)=\frac{1}{2}\sum\limits_{i=1}^nTr_i$.

For a connected graph $G$, the distance matrix of $G$ is denoted by $D(G)=(d_{ij})_{n\times n}$. Let $Tr(G)=\textrm{diag}(Tr_1, Tr_2,\ldots, Tr_n)$ be the diagonal matrix with vertex transmissions of $G$. With the development of the matrix theory of graphs, Aouchiche and Hansen \cite{{AH}} defined the distance signless Laplacian matrix of $G$ as $D^Q(G)= Tr(G) + D(G)$. The largest distance (distance signless Laplacian) eigenvalue of $G$ is called the distance (distance signless Laplacian) spectral radius of $G$, denoted by $\mu(G)$ ($\mu_Q(G)$).

So far, the distance spectral radius and distance signless Laplacian spectral radius of a connected graph have been investigated extensively. For example, in \cite{NP}, Nath and Paul determined the unique graph with the minimum distance spectral radius among all connected bipartite graphs of order $n$ with a given matching number. In \cite{LL}, Lin and Lu determined the graphs with the minimum (maximum) distance signless Laplacian spectral radius
among all connected graphs with fixed order and clique number. In \cite{LZ}, Lin and Zhou determined the unique trees with the second and third maximum distance signless Laplacian spectral radius, respectively, and they also determined the unique connected graph with the second maximum distance signless Laplacian spectral radius. In \cite{QMZL}, Zhang et al. determined the minimum
distance spectral radius of graphs among all $n$-vertex graphs with given connectivity and
independence number, and characterized the corresponding extremal graph. For more results about distance and distance signless Laplacian spectral radius of graphs, readers may refer to \cite{AH2,AP,HR,XZ,XZL,Z} and the references therein.

In \cite{CHT}, Cui et al. proposed to study the convex linear combinations of $Tr(G)$ and $D(G)$, denoted by $D_{\alpha}(G)$ and defined as
$$D_\alpha(G)=\alpha Tr(G)+(1-\alpha)D(G),\ \ \alpha\in[0,1].$$
We call $D_{\alpha}(G)$ as the generalized distance matrix of $G$. It is easy to see $D_0(G)=D(G)$, $D_1(G)=Tr(G)$ and $2D_{\frac{1}{2}}(G)=D^Q(G)$, thus, the matrix $D_{\alpha}(G)$ extends both $D(G)$ and $D^Q(G)$. Therefore, it is very important to study the generalized distance matrix. The largest eigenvalue of $D_{\alpha}(G)$ is called the $D_{\alpha}$ spectral radius or the generalized distance spectral radius of $G$, denoted by $\mu_{\alpha}(G)$.

Recently, the generalized distance spectral radius of graphs has been studied in some papers. In \cite{CTZ}, Cui et al. characterized the unique graph with the minimum generalized distance spectral radius among the connected graphs with fixed chromatic number. In \cite{ABG}, Diaz et al. presented some extremal results on the spectral radius of generalized distance matrix that
generalize previous results on the spectral radii of the distance matrix and distance signless Laplacian matrix. In \cite{LXS}, Lin et al. gave several graph transformations on the $D_{\alpha}$ spectral radius of graphs. There are also some other interesting results on the $D_{\alpha}$ spectral radius of graphs which can be found in \cite{BGGP,ABG,LYH,PGRS}.

In the theory of graph spectra, there is a classic problem which is to establish upper or lower bounds for the related spectral radius of a given class of graphs. For a graph $G$, the adjacent spectral radius of $G$ is denoted by $\rho(G)$. It is well-known that the maximum degree $\Delta(G)$ is no less than $\rho(G)$. In addition, $\Delta(G)=\rho(G)$ if and only if $G$ is a regular graph. Thus, $\Delta(G)-\rho(G)$ has been considered as a measure of irregularity for a graph $G$ and many interesting results regarding its estimation have been obtained, see \cite{CXD,CI,LHY,S,ZHANG}.

By analogy with adjacency matrices, it is natural to consider the difference between the maximum vertex transmission and the distance spectral radius of non-transmission regular graphs. Motivated by some known estimates on $\Delta(G)-\rho(G)$ and some methods used in estimating them, in 2018, Liu et al. \cite{LSX} posed and considered the following questions: How small can $Tr_{max}(G)-\mu(G)$ and  $2Tr_{max}(G)-\mu_Q(G)$ be when $G$ is non-transmission regular? In \cite{LSX}, Liu et al. posed a conjecture about the difference between the maximum
vertex transmission and the distance spectral radius of non-transmission regular graphs. In \cite{LSH}, Liu et al. proved the conjecture proposed by Liu et al. \cite{LSX} and they gave a lower bound for the difference between the maximum vertex transmission and the distance spectral radius of non-transmission regular graphs. In \cite{LL2}, Lan and Liu gave a sharp lower bound of $2Tr_{max}-\mu_Q(G)$ among all $n$-vertex connected graphs, and characterized the extremal graphs.

In this paper we decide to study the difference between the maximum vertex transmission and the $D_\alpha$ spectral radius of connected non-transmission regular graphs, we establish a lower bound on the difference between the maximum vertex transmission and the $D_\alpha$ spectral radius of non-transmission regular graphs, and we also characterize the extremal graphs attaining the bound. Our main result is as follows.

\begin{theorem}\label{T2.6} Let $G$ be a connected non-transmission regular graph of order $n$.
	
	(1) If $n$ is odd and $\alpha\in[0,1)$, then
	
	\begin{equation}
		Tr_{max}-\mu_{\alpha}(G)\ge\frac{(1-\alpha)n+1-\sqrt{[(1-\alpha)n+1]^2-4(1-\alpha)}}{2},
	\end{equation}
	equality holds if and only if $G\cong K_{1,2,2,\cdots,2}$ (the definition see Lemma \ref{L2.4}).
	
	(2) If $n$ is even and $\alpha\in[0,1)$, then
	
	\begin{equation}
		Tr_{max}-\mu_{\alpha}(G)\ge\frac{(1-\alpha)n+2-\sqrt{[(1-\alpha)n+2]^2-8(1-\alpha)}}{2},
	\end{equation}
	equality holds if and only if $G$ is isomorphic to any $(n-4)$-$DVDR$ graph (the definition see Page 4).
	
\end{theorem}

\section{Preliminaries}

In this section, we will introduce some important concepts and useful lemmas which will be used in the following.

\noindent\begin{lemma} (\cite{HJ}) Let $M$ be an irreducible and nonnegative square matrix of order $n$. Let $\rho(M)$ be the spectral radius of $M$. Then the following results hold.\\
		(a) $\rho(M)>0$;\\
		(b) $\rho(M)$ is an algebraically and geometrically simple eigenvalue of $M$;\\
		(c) There is a positive eigenvector corresponding to eigenvalue $\rho(M)$.
\end{lemma}

Equitable quotient matrix has significant applications in calculating the eigenvalues of matrices associated with graphs. Next, we will introduce the definition of the equitable quotient matrix and its some properties.

\noindent\begin{lemma}(\cite{YYSX})\label{D1} Let $M$ be a complex matrix of order $n$ described in the following block form
\begin{equation*}
	M=\begin{bmatrix}
		M_{11} & \cdots & M_{1t} \\
		\vdots& \ddots & \vdots \\
		M_{t1} &\cdots & M_{tt}
	\end{bmatrix},
\end{equation*}
where the blocks $M_{ij}$ are the $n_i\times n_j$ matrices for any $1\le i,j\le t$ and $n=n_1+n_2+\cdots+n_t.$ For $1\le i,j\le t$, let $b_{ij}$ denote the average row sum of $M_{ij}$, i.e. $b_{ij}$ is the sum of all entries in $M_{ij}$ divided by the number of rows. Then $B(M)=(b_{ij})$(or simply $B$) is called the quotient matrix of $M$. If, in addition, for each pair $i$, $j$, $M_{ij}$ has a constant row sum, then $B$ is called the equitable quotient matrix of $M$.
\end{lemma}

\noindent\begin{lemma}\label{L2.3}(\cite{YYSX}) Let $B$ be the equitable quotient matrix of $M$, where $M$ is defined as in Definition \ref{D1}. In addition, let $M$ be a nonnegative matrix, and $\rho(B)$ and $\rho(M)$ denote the spectral radii of $B$ and $M$, respectively. Then $\rho(B)=\rho(M)$.
\end{lemma}

For a connected graph $G$ with order $n$, take $\tau(G)=\frac{Tr_{max}(G)-\mu_{\alpha}(G)}{1-\alpha}$, where $Tr_{max}(G)$ is the maximum vertex transmission of $G$. Then

$$(1-\alpha)\tau(G)=Tr_{max}(G)-\mu_{\alpha}(G).$$

In \cite{AP}, Atik and Panigrahi introduced a class of graphs called $DVDR$ graph. Let $G$ be a connected graph with order $n$. If there exist a vertex $v$ in $G$ such that the degree of $v$ is $  n-1$ and $G-v$ is an regular graph, then we call $G$ to be a distinguished vertex deleted regular graph ($DVDR$, see Figure 1). The vertex $v$ with $d(v)=n-1$ is said to be a distinguished vertex of a graph $G$. In addition, we say $G$ is an $r$-$DVDR$ graph if $G-v$ is $r$-regular.

\begin{figure}[H]
	\begin{centering}
		\includegraphics[scale=0.8]{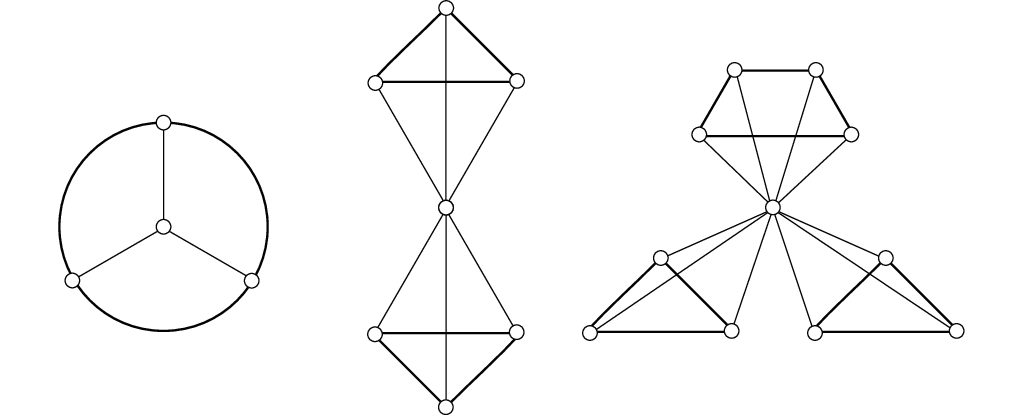}
		\caption{Some examples of $DVDR$ graphs}\label{Fig.1.}
	\end{centering}
\end{figure}

In the following, we will calculate the $(1-\alpha)\tau(G)$ of two special graphs.

\noindent\begin{lemma}\label{L2.4} Let $n=2r-1$ and $K=K_{1,\underbrace{\scriptstyle 2,2,\cdots,2}_{r-1}}$ be the complete $r$-partite graph on $n$ vertices which all subsets contain two vertices except one subset contains only one vertex $v$ with $d(v)=n-1$. Then
	\begin{equation}
		(1-\alpha)\tau(K)=Tr_{max}(K)-\mu_{\alpha}(K)=\frac{(1-\alpha)n+1-\sqrt{[(1-\alpha)n+1]^2-4(1-\alpha)}}{2}.
	\end{equation}
\end{lemma}

\begin{proof}
Let $v$ be the distinguished vertex of $V(K)$. Then the equitable quotient matrix of $D_\alpha(K)$ with respect to the partition ${\{v\}}\cup(V(K)\setminus \{v\})$ is
	\[M=
	\begin{bmatrix}
		\alpha(n-1) & (1-\alpha)(n-1)\\
		1-\alpha & \alpha n+(1-\alpha)(n-1)
	\end{bmatrix}.
	\]
By Lemma \ref{L2.3}, we know that $\mu_{\alpha}(K)$ is equal to the largest eigenvalue of $M$. Hence
	$$\mu_{\alpha}(K)=\frac{(1+\alpha)n-1+\sqrt{[(1-\alpha)n+1]^2-4(1-\alpha)}}{2}.$$
Since $Tr_{max}(K)=n$,
	$$	(1-\alpha)\tau(K)=Tr_{max}(K)-\mu_{\alpha}(K)=\frac{(1-\alpha)n+1-\sqrt{[(1-\alpha)n+1]^2-4(1-\alpha)}}{2},$$
which completes the proof.
\end{proof}

\noindent\begin{lemma} Let $G$ be an $(n-4)$-$DVDR$ graph of order $n$. Let $v$ be the distinguished vertex with $d(v)=n-1$. Then
	\begin{equation}
		(1-\alpha)\tau(G)=Tr_{max}(G)-\mu_{\alpha}(G)=\frac{(1-\alpha)n+2-\sqrt{[(1-\alpha)n+2]^2-8(1-\alpha)}}{2}.
	\end{equation}
\end{lemma}

\begin{proof}
	Similar to the proof of Lemma \ref{L2.4}, the equitable quotient matrix of $D_\alpha(G)$ with respect to the partition ${\{v\}}\cup(V(G)\setminus \{v\})$ is
	\[N=
	\begin{bmatrix}
		\alpha(n-1) & (1-\alpha)(n-1)\\
		1-\alpha & \alpha (n+1)+(1-\alpha)n
	\end{bmatrix}.
	\]
	By Lemma \ref{L2.3}, we know that $\mu_{\alpha}(G)$ is equal to the largest eigenvalue of $N$. Hence
	$$\mu_{\alpha}(G)=\frac{(1+\alpha)n+\sqrt{[(1-\alpha)n+2]^2-8(1-\alpha)}}{2}.$$
	Since $Tr_{max}(G)=n+1$, by a simple calculation , we get
	$$	(1-\alpha)\tau(G)=Tr_{max}(G)-\mu_{\alpha}(G)=\frac{(1-\alpha)n+2-\sqrt{[(1-\alpha)n+2]^2-8(1-\alpha)}}{2},$$
	which completes the proof.
\end{proof}
Based on the above results, we define:
\begin{equation}\label{eq:4}
	(1-\alpha)\tau_n=\frac{(1-\alpha)n+\rho_n-\sqrt{[(1-\alpha)n+\rho_n]^2-4\rho_n(1-\alpha)}}{2},
\end{equation}
where $\rho_n=1$ if $n$ is odd; $\rho_n=2$ if $n$ is even. Clearly, $\tau_n$ satisfies the following equation:
\begin{equation}\label{eq:5}
	(1-\alpha)\tau_n^2-[(1-\alpha)n+\rho_n]\tau_n+\rho_n=0
\end{equation}
It is obvious that if $n$ is odd, then $(1-\alpha)\tau_n=(1-\alpha)\tau(K_{1,2,2,\cdots,2})$; if $n$ is even, then $(1-\alpha)\tau_n=(1-\alpha)\tau(G)$ for any $(n-4)$-$DVDR$ graph $G$.

Let $G_0$ be a graph which attains the minimum of $Tr_{max}(G)-\mu_{\alpha}(G)$ among all connected non-transmission regular graph $G$ of order $n$. Noting that the graphs in Lemma 2.4 and Lemma 2.5 are both connected non-transmission regular graphs and $G_0$ be a graph such that attains the minimum of $(1-\alpha)\tau(G)$ among all connected non-transmission regular graph $G$ of order $n$. Hence $\tau(G_0)\le \tau(G)$. Specifically, $\tau(G_0)\le \tau_n$.

To simplify the proof, we first provide some notations that will be used in the proof. Since $G_0$ is a connected graph, then it follows from the Perron Frobenius Theorem \cite{HJ} that $\mu_{\alpha}(G_0)$ is an eigenvalue of $D_{\alpha}(G_0)$, and there is a unique positive unit eigenvector corresponding to $\mu_{\alpha}(G_0)$.
We use $X=(x_1, x_2, \cdots, x_n)^{T}$ to denote the positive unit eigenvector corresponding to $\mu_{\alpha}(G_0)$. To simplify the writing process, let $x_{max}=\max\{x_i|1\le i\le n\}$ and $x_{min}=\min\{{x_i|1\le i\le n\}}$. Let $Tr_{max}=\max\{Tr(v)|v\in V(G_0)\}$ and $W=W(G_0)=\frac{1}{2}\sum\limits_{w\in V(G_0)}Tr_w$.

In the following sections, we will provide the proof of Theorem 1.1.

\section{Proof of Theorem \ref{T2.6}}

\begin{proof} According to the previous note on $G_0$, let $u$ and $v$ be two vertices of $G_0$ that satisfy $x_u=x_{max}$ and $x_v=x_{min}$, respectively. For convenience, let $\mu_{\alpha}=\mu_{\alpha}(G_0)$. From the $u$-th equation of $D_{\alpha}(G_0)X=\mu_{\alpha}X$, we have
	\begin{align*}
		\mu_{\alpha} x_u&=\alpha Tr_u x_u+(1-\alpha)\sum_{w\in V(G_0)} d(u,w)x_w \\
		&=\alpha Tr_u x_u+(1-\alpha)[\sum_{w\in V(G_0)-\{v\}} d(u,w)x_w+d(u,v)x_v]\\
		&\leq \alpha Tr_{max}x_u+(1-\alpha)[(Tr_{max}-d(u,v))x_u+d(u,v)x_v]\\
		&=Tr_{max}x_u+(1-\alpha)[-d(u,v)x_u+d(u,v)x_v].
	\end{align*}
Based on the above result, then
$$Tr_{max}-\mu_{\alpha}\ge (1-\alpha)d(u,v)(1-\frac{x_v}{x_u})\ge (1-\alpha)(1-\frac{x_v}{x_u}).$$
In addition, according to the properties of $G_0$, it is shown that
\begin{equation}\label{eq:d}
	\tau(G_0)=\frac{Tr_{\max}-\mu_{\alpha}}{1-\alpha}\leq \tau_n.
\end{equation}
Hence, we obtain
	\begin{equation*}\label{eq:b}
		(1-\tau_n)x_u\leq x_v.
	\end{equation*}

{\bf Claim 1.} $\tau_n<1$.

  {\bf Proof.} From (\ref{eq:4}) and $\alpha\in[0,1)$, we consider two cases as follow:\\
Case 1. If $n$ is odd, then $\rho_n$=1. Thus, $\tau_n=\frac{(1-\alpha)n+1-\sqrt{[(1-\alpha)n+1]^2-4(1-\alpha)}}{2(1-\alpha)}$. We take $\tau_n(1)=\tau_n$ under this condition.
If $\tau_n(1)\ge 1$, by a simple calculation, we get $(1-\alpha)^2(n-2)^2\ge (1-\alpha)^2n^2$, this is a contradiction.\\
Case 2. If $n$ is even, then $\rho_n$=2. Thus, $\tau_n=\frac{(1-\alpha)n+2-\sqrt{[(1-\alpha)n+2]^2-8(1-\alpha)}}{2(1-\alpha)}$. We take $\tau_n(2)=\tau_n$ under this condition.
Using the same method as Case 1, we have $\tau_n(2)<1$.

Hence, $\tau_n<1$ and
\begin{equation}\label{eq:8}
\frac{x_{max}}{x_{min}}=\frac{x_u}{x_v}\leq\frac{1}{1-\tau_n}.
\end{equation}

{\bf Claim 2.} $nTr_{max}-2W=\rho_n$ and $\frac{x_{max}}{x_{min}}=\frac{1}{1-\tau_n}$.

{\bf Proof.}
Note that $D_{\alpha}(G_0)X=\mu_{\alpha}X$, we have $\mu_{\alpha}(1,1,\cdots,1)X=(1,1,\cdots,1)D_{\alpha}(G_0)X$. Furthermore,

\begin{equation*}\label{eq:37}
	\mu_{\alpha}\sum\limits_{w\in V(G_0)}x_w=\alpha\sum\limits_{w\in V(G_0)}Tr_wx_w+(1-\alpha)\sum\limits_{w\in V(G_0)}Tr_wx_w=\sum\limits_{w\in V(G_0)}Tr_wx_w.
\end{equation*}
Thus,
\begin{equation}\label{eq:33}
	(Tr_{max}-\mu_{\alpha})\sum_{w\in V(G_0)}x_w=\sum_{w\in V(G_0)}(Tr_{max}-Tr_w)x_w\\
	\ge (nTr_{max}-2W)x_{min},
\end{equation}
and
\begin{equation}
	nTr_{max}-2W\le \frac{(Tr_{max}-\mu_{\alpha})\sum\limits_{w\in V(G_0)}x_w}{x_{min}}.\\
\end{equation}
Observe that $Tr_{max}-\mu_{\alpha}=(1-\alpha)\tau (G_0)\le(1-\alpha)\tau_n$ and $\sum\limits_{w\in V(G_0)}x_w\le x_{min}+(n-1)x_{max}$, combining with (\ref{eq:8}), we obtain
\begin{equation}\label{eq:c}
\begin{aligned}
			nTr_{max}-2W
			&\le \frac{(Tr_{max}-\mu_{\alpha})\sum_{w\in V(G_0)}x_w}{x_{min}} \\
			&\le \frac{(1-\alpha)\tau_n[x_{min}+(n-1)x_{max}]}{x_{min}}\\
			&= (1-\alpha)\tau_n[1+(n-1)\frac{x_{max}}{x_{min}}]\\
			&\le  (1-\alpha)\tau_n[1+(n-1)\frac{1}{1-\tau_n}]\\
			&=(1-\alpha)\tau_n\frac{n-\tau_n}{1-\tau_n}.
\end{aligned}
\end{equation}
From (\ref{eq:5}), we have
\begin{equation}
	\rho_n=\frac{(1-\alpha)(n\tau_n-\tau^2_n)}{1-\tau_n}.
\end{equation}
Therefore $$nTr_{max}-2W\le \rho_n.$$

In what follows, we will check that $nTr_{max}-2W\ge \rho_n$.
Note that $nTr_{max}-2W$ is a positive integer. If $n$ is odd, we have $\rho_n=1$, it is obvious that $nTr_{max}-2W\ge1= \rho_n$. If $n$ is even, we have $\rho_n=2$ and $nTr_{max}-2W$ is an even positive integer.
Thus $nTr_{max}-2W\ge2= \rho_n$.
	
Hence, it can be concluded that $nTr_{max}-2W=\rho_n$. Moreover, all inequalities in (\ref{eq:c}) must be equalities. Thus
$$\frac{x_{max}}{x_{min}}=\frac{1}{1-\tau_n}.$$
	
	{\bf Claim 3.} $Tr_v<Tr_{max}$ and $Tr_u>Tr_{min}$.

	{\bf Proof.} Since $D_{\alpha}(G_0)X=\mu_{\alpha}X$, from the $v$-th equation, we get
	$$\mu_{\alpha}x_v=\alpha Tr_vx_v+(1-\alpha)\sum\limits_{w\in V(G_0)}d(v,w)x_w\ge Tr_vx_v,$$
i.e., $$Tr_v\leq\mu_{\alpha}.$$
Combining with $\mu_{\alpha}<Tr_{max}$, then $Tr_v<Tr_{max}$.

Similarly, we obtain $Tr_u>Tr_{min}$.
	
	{\bf Claim 4.} $Tr_{max}-\mu_{\alpha}=(1-\alpha)\tau_n.$
	
	{\bf Proof.} It follows from (\ref{eq:33}) and Claim 2, we have
\begin{equation}\label{eq:11}
		(Tr_{max}-\mu_{\alpha})\sum_{w\in V(G_0)}x_w
		\ge (nTr_{max}-2W)x_{min}=\rho_nx_{min}.
\end{equation}
Hence
	\begin{equation}\label{eq:12}
		Tr_{max}-\mu_{\alpha}\ge\frac{\rho_n x_{min}}{\sum\limits_{w\in V(G_0)}x_w}\ge\frac{\rho_n x_{min}}{x_{min}+(n-1)x_{max}}=\frac{\rho_n}{1+(n-1)\frac{x_{max}}{x_{min}}}.
	\end{equation}
From Claim 2, we know that $\frac{x_{max}}{x_{min}}=\frac{1}{1-\tau_n}$, then
	\begin{equation}\label{eq:13}
		Tr_{max}-\mu_{\alpha}\ge\frac{\rho_n}{1+(n-1)\frac{1}{1-\tau_n}}=\frac{\rho_n(1-\tau_n)}{n-\tau_n}.
	\end{equation}
In addition, from (\ref{eq:5}), we get
$$\rho_n(\tau_n-1)=(1-\alpha)(\tau^2_n-n\tau_n).$$
Furthermore, (\ref{eq:13}) can be written simply as
\begin{equation}\label{eq:17}
	Tr_{max}-\mu_{\alpha}\ge\frac{\rho_n(1-\tau_n)}{n-\tau_n}=(1-\alpha)\tau_n.
\end{equation}	
	
By (\ref{eq:d}), noting that $Tr_{max}-\mu_{\alpha}\le(1-\alpha)\tau_n$.

Hence

$$Tr_{max}-\mu_{\alpha}=(1-\alpha)\tau_n,$$
and all inequalities above must be equalities.
	
{\bf Claim 5.}  $Tr_{min}=n-1$ and $Tr_{max}=Tr_{min}+\rho_n=n-1+\rho_n.$

{\bf Proof.}
According to the proof of Claim 4, we know that all the inequalities in Claim 4 must be equalities.

	 From the equality in (\ref{eq:12}), it can be concluded that $x_w=x_{max}$ for $w\in V(G_0)-\{v\}$. Since $x_v=x_{min}$, from the equality in (\ref{eq:11}), we have
	$$\sum_{w\in V(G_0)}(Tr_{max}-Tr_w)x_w=\sum_{w\in V(G_0)-\{v\}}(Tr_{max}-Tr_w)x_{max}+(Tr_{max}-Tr_v)x_v=\rho_nx_{min}.$$

Next, we will proof that $Tr_w=Tr_{max}$ for ${w\in V(G_0)-\{v\}}$.

 If it does not true, we have $\sum\limits_{w\in V(G_0)-\{v\}}(Tr_{max}-Tr_w)x_{max}\ge x_{max}$. From Claim 3, we know $(Tr_{max}-Tr_v)x_v\ge x_v=x_{min}.$ Since $G_0$ is non-transmission regular, then $x_{max}>x_{min}$. Hence $$\sum_{w\in V(G_0)-\{v\}}(Tr_{max}-Tr_w)x_{max}+(Tr_{max}-Tr_v)x_v\ge x_{max}+x_{min}>2x_{min}\ge \rho_nx_{min}.$$
	This is a contradiction.
	Hence, we get that $Tr_w=Tr_{max}$ for ${w\in V(G_0)-\{v\}}$ and
	\begin{equation}\label{eq:15}
		Tr_{max}-Tr_v=Tr_{max}-Tr_{min}=\rho_n.
	\end{equation}
Furthermore
\begin{align*}
	\mu_{\alpha}{x_{min}}=\mu_{\alpha}x_v&=\alpha Tr_vx_v+(1-\alpha)\sum\limits_{w\in V(G_0)}d(v,w)x_w\\
	&=\alpha Tr_{min}{x_{min}}+(1-\alpha)Tr_{min}x_{max}.
\end{align*}
From Claim 2, $\frac{x_{max}}{x_{min}}=\frac{1}{1-\tau_n}$, then
\begin{equation}\label{eq:16}
	\mu_{\alpha}=\alpha Tr_{min}+(1-\alpha)Tr_{min}\frac{x_{max}}{x_{min}}=\alpha Tr_{min}+(1-\alpha)Tr_{min}\frac{1}{1-\tau_n}.
\end{equation}
From the equality in (\ref{eq:13}), (\ref{eq:15}) and (\ref{eq:16}), it can be concluded that
\begin{align*}
	\rho_n\frac{1-\tau_n}{n-\tau_n}&=Tr_{max}-\mu_{\alpha}\\
	&=Tr_{max}-\alpha Tr_{min}-(1-\alpha)Tr_{min}\frac{1}{1-\tau_n}\\
	&=\rho_n+(1-\alpha)Tr_{min}-(1-\alpha)Tr_{min}\frac{1}{1-\tau_n},
\end{align*}
which implies $Tr_{min}=\frac{1}{1-\alpha}\frac{n-1}{n-\tau_n}\frac{1-\tau_n}{\tau_n}\rho_n$.

Recall that $\tau_n$ satisfies the equation (\ref{eq:5}), we get $\rho_n=\frac{(1-\alpha)(\tau^2_n-n\tau_n)}{\tau_n-1}.$
Then
$$\begin{cases}
	Tr_{min}=\frac{1}{1-\alpha}\frac{n-1}{n-\tau_n}\frac{1-\tau_n}{\tau_n}\rho_n\\
	\rho_n=\frac{(1-\alpha)(\tau^2_n-n\tau_n)}{\tau_n-1}
\end{cases}.$$
By a simple calculation, we get $Tr_{min}=n-1$.

Furthermore, by (\ref{eq:15}), we have $Tr_{max}=Tr_{min}+\rho_n=n-1+\rho_n$, which completes the proof of this claim.

According to the proof process mentioned above, we continue to prove Theorem 1.1.

From Claim 5, we get $Tr_{min}=Tr_v=n-1$, which implies $d(v)=n-1$ and $diam(G_0)=2$.
For any $y\in V(G_0)$,
$$Tr_y=\sum\limits_{x\in V(G_0)}d(y,x)=d(y)+2(n-1-d(y))=2n-2-d(y).$$
Hence, for $z\in V(G_0)-\{v\}$, we have
$d(z)=2(n-1)-Tr_z=2n-2-(n-1+\rho_n)=n-1-\rho_n$,
which implies that if $n$ is odd, $d(z)=n-2$; if $n$ is even, $d(z)=n-3.$
Hence, if $n$ is odd, then $G_0\cong K_{1,2,\cdots,2}$; if $n$ is even, then $G_0$ is isomorphic to any $(n-4)$-$DVDR$ graph.

Note that $G_0$ is a graph which attains the minimum $Tr_{max}(G)-\mu_{\alpha}(G)$ among all connected non-transmission regular graph $G$ of order $n$. Hence, for any connected non-transmission regular graph $G$ of order $n$, it can be concluded that:
$$Tr_{max}(G)-\mu_{\alpha}(G)\ge Tr_{max}(G_0)-\mu_{\alpha}(G_0)=(1-\alpha)\tau_n.$$

	According to the above process, we have:

	(1) If $n$ is odd and $\alpha\in [0,1)$, then
	
	\begin{equation*}
		Tr_{max}-\mu_{\alpha}(G)\ge(1-\alpha)\tau_n(1)=\frac{(1-\alpha)n+1-\sqrt{[(1-\alpha)n+1]^2-4(1-\alpha)}}{2},
	\end{equation*}
where $\tau_n(1)=\frac{(1-\alpha)n+1-\sqrt{[(1-\alpha)n+1]^2-4(1-\alpha)}}{2(1-\alpha)}$. Equality holds if and only if $G\cong K_{1,2,\cdots,2}$.

	(2) If $n$ is even and $\alpha\in[0,1)$, then
	
	\begin{equation*}
		Tr_{max}-\mu_{\alpha}(G)\ge(1-\tau_n(2))=\tau_n(2)=\frac{(1-\alpha)n+2-\sqrt{[(1-\alpha)n+2]^2-8(1-\alpha)}}{2},
	\end{equation*}
where $\tau_n(2)=\frac{(1-\alpha)n+2-\sqrt{[(1-\alpha)n+2]^2-8(1-\alpha)}}{2(1-\alpha)}$. Equality holds if and only if $G$ is isomorphic to any $(n-4)$-$DVDR$ graph.
\end{proof}

\section*{Statement}

This article has undergone further revisions and enhancements, all of which were contributed by Ligong Wang. In recognition of his contributions, Ligong Wang is now acknowledged as a new co-author. 
The manuscript has been updated to its current version to incorporate these changes. We affirm that all authors have reviewed and approved this update.

\end{document}